\documentclass[a4paper, 11pt]{article}
\usepackage{verbatim}

\usepackage{amsmath,amssymb}
\newtheorem{theorem}{Theorem}

\newtheorem{corollary}[theorem]{Corollary}

\newtheorem{definition}[theorem]{Definition}

\newtheorem{lemma}[theorem]{Lemma}

\newtheorem{proposition}[theorem]{Proposition}

\newenvironment{proof}[1][Proof]{\noindent\textbf{#1.} }{\ \rule{0.5em}{0.5em}}
\topmargin by -\headsep
\newcommand{\R}{\mathbb{R}}
\newcommand{\E}{\mathbb{E}}
\renewcommand{\P}{\mathbb{P}}
\newcommand{\bA}{\mathbf{A}}

\newcommand{\cP}{\mathcal{P}}
\newcommand{\X}{\mathbf{X}}
\renewcommand{\S}{\mathbb{S}}

\title{Asymptotically Optimal Multi-Paving}
\author{Mohan Ravichandran\footnote{Mimar Sinan Fine Arts University, \texttt{mohan.ravichandran@msgsu.edu.tr}, Supported
by Tubitak 1001 grant 115F204} \hspace{0.1in} and \hspace{0.1in} Nikhil
Srivastava\footnote{University of California, Berkeley, \texttt{nikhil@math.berkeley.edu}, Supported by NSF grant
CCF-1553751 and a Sloan research fellowship.}}
\date{}
\begin{document}
\maketitle
\begin{abstract}
Anderson's paving conjecture, now known to hold due to the resolution of the
Kadison-Singer problem asserts that every zero diagonal Hermitian matrix admits
non-trivial pavings with dimension independent bounds. In this paper, we develop
a technique extending the arguments of Marcus, Spielman and Srivastava in their
solution of the Kadison-Singer problem to show the existence of non-trivial 
pavings for collections of matrices. We show that given zero diagonal Hermitian
contractions $A^{(1)}, \cdots, A^{(k)} \in M_n(\mathbb{C})$ and $\epsilon > 0$,
one may find a paving $X_1 \amalg \cdots \amalg X_r = [n]$ where $r \leq
18k\epsilon^{-2}$ such that, \[\lambda_{max} (P_{X_i} A^{(j)} P_{X_i}) <
\epsilon, \quad i \in [r], \, j \in [k].\] 
As a consequence, we get the correct asymptotic estimates for paving general zero diagonal matrices; zero diagonal contractions can be $(O(\epsilon^{-2}),\epsilon)$ paved. As an application, we give a simplified proof wth slightly better estimates of a theorem of Johnson, Ozawa and Schechtman concerning commutator representations of zero trace matrices.
\end{abstract}

\section{Introduction}
Anderson's Paving conjecture \cite{anderson1979extreme} asserts that for every $\epsilon>0$ there is an
integer $r$ such that every zero diagonal complex $n\times n$ matrix $M$ admits
a {\em paving} $X_1\cup\ldots\cup X_{r}=[n]$, which satisfies: $$\|P_{X_i} M
P_{X_i}\|< \epsilon\,\|M\| \quad\textrm{for all $i=1,\ldots,r,$}$$ 
where the $P_{X_i}$ are diagonal projections.
This conjecture, which implies a positive solution to the Kadison-Singer Problem,
is proven in \cite{MSS2} with the dependence $r=(6/\epsilon)^8$ which arises as
follows. First, the ``one-sided'' bound 
$$\lambda_{max}(P_{X_i}MP_{X_i})< \epsilon\, \|M\|$$
is obtained with $r=(6/\epsilon)^2$ for {\em Hermitian}
$M$, via the method of interlacing families of polynomials. Then this is
extended to a two-sided bound by taking a product of pavings for $M$ and $-M$,
and finally to non-Hermitian $M=A+iB$ by taking yet another product of two-sided
pavings of $A$ and $B$. Each product increases $r$ quadratically, and in general
this approach allows one to simultaneously pave any $k$ Hermitian matrices in
the one-sided sense with the dependence $(6/\epsilon)^{2k}$. Our main theorem says,


\begin{theorem}\label{thm:paving}
Given zero diagonal Hermitian contractions $A^{(1)}, \cdots, A^{(k)} \in
M_n(\mathbb{C})$ and $\epsilon > 0$, there exists a paving $X_1 \amalg \cdots \amalg X_r = [n]$ where $r \leq 18k\epsilon^{-2}$ such that,
\[\lambda_{max} (P_{X_i} A_j P_{X_i}) < \epsilon, \quad i \in [r], \, j \in [k].\]
\end{theorem}
Thus, we improve the dependence for simultaneously (one-sided) paving $k$ matrices from
exponential in $k$ to linear in $k$. 
Taking $k=2$ with the pair $A, -A$ for a zero diagonal Hermitian $A$ yields two
sided pavings of size $r = O(1/\epsilon^2)$.

We remark that this bound matches the $r=\Omega(1/\epsilon^2)$ lower bound
established by Casazza et al. \cite{CEKP07}, and the $r=\Omega(k)$ lowerbound
established by Popa and Vaes \cite{popa2015paving}, who also speculated as to
whether a polynomial or linear dependence might be achievable. In the last
section, we show how one may combine the above two examples to construct $k$
tuples of zero diagonal matrices for which a multi-paving requires at least $k
\lfloor \epsilon^{-2}\rfloor$ blocks. Thus, our result is asymptotically optimal
in terms of both $k$ and $\epsilon$. 

We note that in the infinite-dimensional setting, Popa and Vaes \cite{popa2015paving} have shown that given a singular MASA $\mathcal{A}$ in a type $II_1$ factor $\mathcal{M}$, when considering the tuple $\mathcal{A}^{\omega} \subset \mathcal{M}^{\omega}$ where $\mathcal{A}^{\omega}$ and $\mathcal{M}^{\omega}$ denote the ultrapowers of the relevant von Neumann algebras, the bound $\epsilon=2\sqrt{r-1}/r$ is achievable for
$(r,\epsilon)-$paving for any finite collection of zero diagonal operators . We
emphasise here that Popa and Vaes show that in this setting, the multi-paving bound is \emph{independent} of the number of operators. However, when it comes to MASAs with large normalisers, or finite matrices for that matter, the situation is quite different; the number of blocks must grow linearly in the number of matrices.

A corollary of our main theorem is the following improvement of the
quantitative commutator theorem of \cite{JOS} (who had provided an estimate of
$e^{K\sqrt{\log m\log\log m}a}$ for a large constant $K$), whose main advantage is a significantly simplified proof.
\begin{corollary}
Every zero trace matrix $A \in M_n(\mathbb{C})$ may be written as $A = [B, C]$, such that $\|B\|\|C\|\le 300\, e^{9\sqrt{\log(m)}}||A||$.
\end{corollary}
We include a short proof of this in the final section of the paper.

\subsubsection*{Restricted Invertibility}
Restricted invertibility refers to the problem of choosing a single submatrix
(as opposed to a paving) of a given matrix with small norm. The first such
result was given by Bourgain and Tzafriri in \cite{BT87}, with improvements and
generalizations by \cite{vershynin2001john, SS12, naor2016restricted, MRI, IF3}.
We prove the following result on choosing a single submatrix which
simultaneously has small norm for a given collection of matrices
$A^{(1)},\ldots,A^{(k)}$.
Our bounds have the merit of
being slightly stronger (by a constant factor) than those obtained by applying paving and choosing the
largest part. The proof is also somewhat simpler than that of Theorem
\ref{thm:paving}.
\begin{theorem}\label{thm:ri}
Given zero diagonal Hermitian contractions $A^{(1)}, \cdots, A^{(k)} \in
M_n(\mathbb{C})$ and $\epsilon > 0$, one may find a subset $\sigma \subset [n]$
of size $n\epsilon^2/6k$ such that, \[\lambda_{max} (P_{\sigma} A_j
P_{\sigma}) < \epsilon, \quad j \in [k].\]
\end{theorem}

For two sided restricted invertibility for Hermitian matrices, we get the estimate $n\epsilon^2/12$, 
which improves the estimates of $n\epsilon^2 (\sqrt{2}-1)^4/2 \approx n\epsilon^2/66$ due to Pierre Youssef \cite{YoussefIMRN}, which however, had the important merit of being constructive. 

\subsubsection*{Techniques and Organization}
To introduce our techniques, let us briefly recall how \cite{MSS2} constructs
pavings. First, a one-sided paving is obtained via the method of interlacing
families of polynomials (see e.g. \cite{MSSICM} for an overview), which relies on three facts:
\begin{enumerate}
\item [(I)] The largest eigenvalue of a Hermitian matrix is equal to the largest root of its
characteristic polynomial.
\item [(II)]The characteristic polynomials of the matrices:
$$\sum_{i\le r} P_{X_i}MP_{X_i},\quad {X_1\cup\ldots\cup X_r=[n]}$$
form an {``interlacing family''}, which implies in particular, that their average is real-rooted and
that there exists a polynomial in the family whose largest root is at most that
of the sum.
\item [(III)] The largest root of the average polynomial can be bounded using a
``barrier function argument''.
\end{enumerate}
This method is only able to control the largest eigenvalue of a matrix, and
therefore can only give one-sided bounds. The suboptimal exponential
dependence on $k$ for simultaneously paving $k$ matrices arises from
sequentially applying the one-sided result in a black box way $k$ times.

The main contribution of this paper is a technique for simultaneously carrying out the
interlacing argument {``in parallel''} for $k$ matrices, losing only a 
factor of $k$ in the process. The key idea is to replace the determinant of a
single matrix by a {\em mixed determinant} of $k$ matrices.

\begin{definition}\label{def:mixdet}
Given a $k$ tuple of matrices $\textbf{A} = (A^{(1)}, \cdots, A^{(k)})$ in $M_n(\mathbb{C})$, the mixed determinant is defined as, 
\[D[\textbf{A}] := \sum_{S_1 \amalg \cdots \amalg S_k = [n]} \operatorname{det}[A^{(1)}(S_1)] \cdots \operatorname{det}[A^{(k)}(S_k)].\]
\end{definition}
Closely related to the mixed determinant is the following generalization of the
characteristic polynomial of a single matrix to a tuple of $k$ matrices, which
we call the \emph{mixed determinantal polynomial} (introduced in \cite{MRKSP}) or \emph{MDP} in short. 

\begin{definition}\label{def:mdp}
Given a $k$ tuple of matrices $\textbf{A} = (A^{(1)}, \cdots, A^{(k)})$ in $M_n(\mathbb{C})$, the mixed determinantal polynomial is defined as, 
\[\chi[\textbf{A}](x):= k^{-n}D\left[xI-A^{(1)}, \cdots, xI-A^{(k)}\right].\]

\end{definition}

Closely related polynomials were also used by Borcea and Branden in \cite{BBJohnson} in their solution of Johnson's conjectures. This definition is useful because it turns out that the largest eigenvalues of $k$ matrices are
{\em simultaneously} controlled by the largest root of their MDP up to a factor
of $k$, serving as a substitute for (I).
\begin{theorem}\label{thm:shrink}
 Let $A^{(1)}, \cdots, A^{(k)}$ be zero diagonal real symmetric matrices. Then, 
\[\operatorname{max}_{i \in [k]} \lambda_{max}\,\chi(A^{(i)}) \leq
k\,\lambda_{max} \,\chi[A^{(1)},\ldots,A^{(k)}] .\]
\end{theorem}
We prove Theorem \ref{thm:shrink} in Section \ref{sec:shrink}.

We then show that parts (II) and (III) of the interlacing families argument can
be carried out for MDPs. As in previous works, (II) is established by deriving a
differential formula for MDPs which shows that they and all of their relevant
convex combinations are real-rooted. For joint restricted invertibility, (III)
is derived from the univariate root shrinking estimates of \cite{MRI}; for
multi-paving, the required bound follows by observing that the relevant
expected characteristic polynomials are equivalent (after a change of variables)
to the mixed characteristic polynomials of \cite{MSS2}, and appealing to the bounds
derived there.

The interlacing families for joint restricted invertibility and paving are
analyzed in Sections \ref{sec:ri} and \ref{sec:paving} respectively, proving
Theorems  \ref{thm:ri} and \ref{thm:paving}.

\section{Preliminaries}
\subsubsection*{Submatrices, Pavings, and Derivatives}
We will use the notation $A(S)$ to denote the submatrix of $A$ indexed by rows
and columns in $S$, and $A_{S}$ to denote the matrix with rows and columns in
$S$ removed. For a collection of subsets $\mathbf{S} = S_1\cup\ldots S_k$ we will use
$$A(\mathbf{S}):=\bigoplus_{i} A(S_i)$$
to denote the block matrix containing submatrices indexed by the $S_i$.

Together with the \emph{mixed determinantal polynomial} from the introduction, we will also need a closely related polynomial, 
\begin{definition}
Let $\textbf{A}  = (A^{(1)}, \cdots, A^{(k)})$   be matrices in $M_n(\mathbb{C})$ and let $S \subset [n]$. Define the restricted mixed determinantal polynomial by, 
\[\chi[\textbf{A}_{S} ]:= \chi[A^{(1)}_S, \cdots, A^{(k)}_S],\]
and analogously
\[\chi[\textbf{A}(S) ]:= \chi[A^{(1)}(S), \cdots,
A^{(k)}(S)]\]
\end{definition}

Central to our main result will be characteristic polynomials and mixed determinantal polynomials of pavings.
\begin{definition}
Let $A \in M_n(\mathbb{C})$ and let $\textbf{S}  = \{S_1, \cdots, S_r\}$  be a collection of subsets  of $[n]$, that we will typically take to be a partition of $[n]$. Then, we define the characteristic polynomial of the collection, 
\[\chi[A(\textbf{S})] := \prod_{i = 1}^{r} \chi[A(S_i)] .\]
Also, given matrices $\textbf{A}  = (A^{(1)}, \cdots, A^{(k)})$  in $M_n(\mathbb{C})$, we define the mixed determinantal polynomial of the collection, 
\[\chi[\textbf{A}(\textbf{S})] := \prod_{i = 1}^{r} \chi[\textbf{A}(S_i)] =
\prod_{i = 1}^{r} \chi[A^{(1)}(S_i), \cdots, A^{(k)}(S_i)] ,\]
noting that the paving $\mathbf{S}$ takes precedence over the tuple $\bA$.
\end{definition}

We will use multiset notation to indicate partial derivatives, i.e., for a
multiset $S=(\kappa_1,\ldots,\kappa_n)$ containing elements of $[n]$ with
frequencies $\kappa_i$, we will write
$$\partial^S:=\prod_{i\le n}\partial_{z_i}^{\kappa_i}$$
where the variable $z_i$ will be clear from the context. Multiplying a multiset
by an integer means multiplying the multiplicity of each element in it, for
example $\partial^{k[n]}$ means every element in $[n]$ with multiplicity $k$.
We will frequently deal with doubly indexed arrays of indeterminates
$z^{(i)}_j$. Let $Z_i = \operatorname{diag}(z^{(i)}_{1}, \cdots, z^{(i)}_{n})$ for $i \in [r]$ be a tuple of diagonal matrices of indeterminates.
For a multiset $S$,  we will use the shorthand notation 
\[\partial^{S}_{(i)} = \prod_{j \in S}\partial_{z^{(i)}_j},  i \in [r]\]
to indicate differentiation with respect to the $(i)$ variables.

We will make frequent use of differential formulas for mixed determinantal
polynomials and their variants.
\begin{proposition}
Let $\textbf{A}  = (A^{(1)}, \cdots, A^{(k)})$   be matrices in $M_n(\mathbb{C})$. Then, 
\begin{eqnarray}\label{MDP}\chi[\textbf{A}] = \dfrac{1}{(k!)^n}\partial^{(k-1)[n]}\prod_{j = 1}^{k}\operatorname{det}[Z-A^{(j)}]\mid_{Z = xI}.
\end{eqnarray}
Suppose $S \subset [n]$. Then, 
\begin{eqnarray}\label{restrictedMDP}
\chi[\textbf{A}(S)] =\dfrac{1}{(k!)^{n}} \partial^{k[n]\setminus S}\prod_{j = 1}^{k}\operatorname{det}[Z-A^{(j)}] \mid_{Z = xI}.
\end{eqnarray}
Let $\textbf{S}  = \{S_1, \cdots, S_r\}$  be a partition of $[n]$. Then, the mixed determinantal polynomial of the paving may be written as, 
\begin{eqnarray}\label{MDPPaving}
\chi[\textbf{A}(\textbf{S})] = \dfrac{1}{(k!)^{rn}}\left(\prod_{i \in [r]} \partial_{(i)}^{k[n]\setminus S_i}\right) \prod_{i = 1}^{r} \prod_{j = 1}^{k}\operatorname{det}[Z_i - A^{(j)}]\mid_{Z_1 = \cdots = Z_r = xI}
\end{eqnarray}
\end{proposition}
\begin{proof}
Observe that for $A \in M_n(\mathbb{C})$ and $S \subset [n]$, we have,
\[\chi[A(S)](x) = \partial^{[n]\setminus S} \operatorname{det}[Z-A] \mid_{Z = xI}.\]
By the Leibnitz formula, we have that, 
\begin{eqnarray*}
\partial^{(k-1)[n]}  \left(\prod_{i = 1}^{k} \operatorname{det}[Z-A^{(i)}] \right)&=& [k-1!]^n\sum_{T_1 \amalg \cdots \amalg T_k = (k-1)[n]}\prod_{i = 1}^k \partial^{T_i}\operatorname{det}[Z-A^{(i)}]
\end{eqnarray*}
Given a partition $T_1 \amalg \cdots \amalg T_k = (k-1)[n]$, for the corresponding term above to be non-zero, all the $S_i$ must be subsets of $[n]$, which in turn yields that yields that letting $S_i = [n]\setminus T_i$ for $i \in [k]$, that $S_1 \amalg \cdots \amalg S_k= [n]$. We therefore have that, 
\begin{eqnarray*}
\partial^{(k-1)[n]}  \left(\prod_{i = 1}^{k} \operatorname{det}[Z-A^{(i)}] \right)\mid_{Z = xI}&=& [(k-1)!]^n\sum_{S_1 \amalg \cdots \amalg S_k = [n]}\prod_{i = 1}^k \operatorname{det}[(Z-A^{(i)})(S_i)]\mid_{Z = xI}\\
&=&(k!)^n  k^{-n}\sum_{S_1 \amalg \cdots \amalg S_k = [n]} \chi[A^{(1)}(S_1)] \cdots \chi[A^{(k)}(S_k)]\\
&=& (k!)^n \chi[\textbf{A}]
\end{eqnarray*}
For (\ref{restrictedMDP}), we observe that, 

\begin{eqnarray*}
\chi[\textbf{A}(S)] &=& \chi[A^{(1)}(S), \cdots, A^{(k)}(S)]\\
&=& \dfrac{1}{(k!)^{|S|}}\partial^{(k-1)S}\prod_{j = 1}^{k}\operatorname{det}[(Z-A^{(j)})(S)]\mid_{Z = xI}\\
&=& \dfrac{1}{(k!)^{|S|}}\partial^{(k-1) S}\prod_{j = 1}^{k}\partial^{[n]\setminus S}\operatorname{det}[Z-A^{(j)}]\mid_{Z = xI}\\
&=& \dfrac{1}{(k!)^n}\partial^{k[n]\setminus S}\prod_{j = 1}^{k}\operatorname{det}[Z-A^{(j)}]\mid_{Z = xI},
\end{eqnarray*}
where in the last line we have used the fact that $\det[Z-A^{(j)}]$ is
multilinear in the variables $z_i$.

And finally, (\ref{MDPPaving}) follows from observing that, 
\[\chi[\textbf{A}(\textbf{S})] = \prod_{i = 1}^{r} \chi[\textbf{A}(S_i)].\]
\end{proof}

\subsubsection*{Interlacing, Stability, and Mixed Characteristic Polynomials}
We will use the following facts about real-rooted polynomials.

\begin{definition}\label{def:interlacing}
We say that a real rooted
  polynomial $g(x) = \alpha_{0} \prod_{i=1}^{n-1} (x - \alpha_{i})$ \emph{interlaces} a
  real rooted polynomial 
  $f(x) = \beta_{0} \prod_{i=1}^{n} (x - \beta_{i})$ if
\[
  \beta_{1} \leq \alpha_{1} \leq \beta_{2} \leq \alpha_{2} \leq \dotsb \leq
  \alpha_{n-1}\leq \beta_{n}.
\]
We say that polynomials $f_{1} , \dotsc , f_{k} $ have a \emph{common interlacer}
  if there is a polynomial $g $ so that $g $ interlaces $f_i $ for each $i$. \end{definition}

By a result of Fell \cite{Fell}, having a common interlacer is equivalent to asserting that for every
$f_i$ and $f_j$ all of the convex combinations $\alpha f_i+(1-\alpha)f_j$ are
real-rooted.
We recall the following elementary lemma from \cite{MSS1} that shows the utility of having a common interlacing.

\begin{lemma}\label{lem:interlacing}
Let $f_{1}, \dotsc , f_{k}$ be polynomials of the same degree that are real-rooted and have positive leading coefficients.
Define
\[
  f_{\emptyset} = \sum_{i=1}^{k} f_{i}.
\]
If $f_{1}, \dotsc , f_{k}$ have a common interlacing,
  then
  there exists an $i$ so that the largest root of $f_{i}$ is at most the largest root of
  $f_{\emptyset}$.
\end{lemma}

An important class of real-rooted polynomials, which are closely related to
MDPs, are {\em mixed characteristic polynomials}.

\begin{definition}\label{def:mcp} Given independent random vectors
$r_1,\ldots,r_m$ with covariance matrices $A_i:=\E r_ir_i^*$, the {\em mixed
characteristic polynomial} of $A_1,\ldots,A_m$ is defined as:
$$\mu(A_1,\ldots,A_m)[x] := \E \chi\left(\sum_{i=1}^m r_ir_i^*\right).$$
\end{definition}
It was shown in \cite{MSS2} that the expression on the right hand side only depends on the
$A_i$ and not on the details of the $r_i$, so mixed characteristic polynomials
are well-defined. We will exploit the following key property of mixed characteristic
polynomials, also established in \cite{MSS2}:
\begin{theorem}\label{thm:if2} If $A_1,\ldots,A_m$ are positive semidefinite
matrices then $\mu(A_1,\ldots,A_m)[x]$ is real-rooted. If $\sum_{i}A_i=I$ and
$\mathrm{Tr}(A_i)\le\epsilon$ for all $i$ then
$$ \lambda_{max}\,\mu(A_1,\ldots,A_m) < (1+\sqrt{\epsilon})^2.$$
\end{theorem}

Recall that a polynomial $p(z_1,\ldots,z_n)$ is {\em real stable} if its
coefficients are real and it has no zeros with $\mathrm{Im}(z_i)>0$ for all $i$.
Note that a univariate real stable polynomial is real-rooted, and that a degree
one polynomial with coefficients of the same sign is always real stable.

We will rely on the following theorem of Borcea and Branden, which characterizes
differential operators preserving stability.
\begin{theorem}[Theorem 1.3 in \cite{BBWeylAlgebra}] \label{thm:bbweyl}
Let $T : \R[z_1, \dots , z_n ] \to  \R[z_1, \dots , z_n ]$ be an operator of the form
\[
T = \sum_{\alpha , \beta \in \mathbb{N}^n} c_{\alpha, \beta} z^\alpha \partial^{\beta}
\]
where 
 $c_{\alpha, \beta} \in \mathbb{R}$ and $c_{\alpha ,\beta}$
  is zero for all but finitely many terms.
Define the {\em symbol} of $T$ to be
\[ 
F_T(z,w) := \sum_{\alpha, \beta} c_{\alpha, \beta} z^{\alpha}w^{\beta}.
\]
Then $T$ preserves real stability if and only if $F_T(z, -w)$ is real stable.
\end{theorem}

Finally, we will use the elementary symmetric functions 
$$e_i(\lambda_1,\ldots,\lambda_n) = \sum_{|S|=i}\prod_{j\in S}\lambda_j,$$
and the notation $e_i(A)$ to denote the $i^{th}$ coefficient of the characteristic
polynomial of a matrix $A$, which is equal to $e_i$ of its eigenvalues.

\section{Joint Control of The Largest Root}\label{sec:shrink}
Theorem \ref{thm:shrink} is a consequence of the following monotonicity result.
\begin{theorem}\label{thm:matrixroot}
Let $\textbf{A} = (A^{(1)}, \cdots, A^{(k)})$ be a $k$ tuple of Hermitian
matrices and let $B$ be a zero diagonal Hermitian matrix. Then, 
\[\lambda_{max} \chi[\textbf{A}, B] \geq \lambda_{max} \chi[\textbf{A}, 0].\]
\end{theorem}
We will deduce this by setting the rows and columns of $B$ to zero one by one. Let
$B_{(S)}$ denote the $n \times n$ matrix
gotten from $B$ by setting entries in rows and columns from $S\subseteq [n]$ to zero.

\begin{proposition}\label{principal}
Let $\textbf{A} = (A^{(1)}, \cdots, A^{(k)})$ and $B$ be zero diagonal $n\times
n$ Hermitian matrices. Then, 
\[\lambda_{max} \chi[\textbf{A}, B] \geq \lambda_{max} \chi[\textbf{A}, B_{(1)}].\]
\end{proposition}
\begin{proof}
Let $B_{t}$ denote the matrix $B_t = D_t B D_t$ where $D_t =
\operatorname{diag}(\sqrt{t}, 1, \cdots, 1)$. Since 
$$p_t:=\chi[\textbf{A}, B_{t}]$$
is an MDP it is real-rooted for all $t\ge 0$.  
Let $f:[0,\infty)\rightarrow \R$ by $$f(t):=\lambda_{max}(p_t).$$ We will show
that $f$ is monotone increasing in $t$.

Observe that $p_t$ is a polynomial of degree at most one in $t$ since by Definition \ref{def:mixdet}
it is a sum over products of characteristic polynomials of principal submatrices
of the $A^{(i)}$ and $B$, which either contain both
the first row and column of $B_t$ (introducing terms containing a factor of $t$) or
contain neither.  Thus, we may write
$p_t(x)=r(x)+ts(x)$
for some polynomials $r,s$. Moreover, $r(x)=p_0(x)$ is real-rooted
and $s(x)=\lim_{t\rightarrow \infty}p_t(x)$ must be real-rooted or identically zero by
Hurwitz's theorem; let us assume the former case since otherwise we are done.

Since the largest root of a real-rooted polynomial is continuous in its
coefficients, $f$ is continuous. Assume for a moment that $r(x)$ and $s(x)$ have no
common roots. If $f$ is not monotone, then by continuity we can choose two points
$t_1\neq t_2$ such that $f(t_1)=f(t_2)=z$. This implies that 
$$ r(z)+t_1s(z)=r(z)+t_2s(z)=0,$$
whence $s(z)=0$, and consequently $r(z)=0$, contradicting that $r(x)$ and $s(x)$ have
no common roots. 

To handle the general case, let $r(x)=q(x)r_1(x)$ and $s(x)=q(x)s_1(x)$ where
$r_1$ and $s_1$ have no common factors. Observe that every root of $q$ is a root
of $p_t$ for all $t$. Thus, 
$$f(t)=\lambda_{max}(r(x)+ts(x)) = \max\{M, \lambda_{max}(r_1(x)+ts_1(x))\},$$
where $M$ is the largest root of $q$, which is monotone by the preceding argument.

We now show that $f$ is nondecreasing in $t$ by examining its behavior at
$\infty$. Assume that the first row of $B$ is nonzero since otherwise we are done. Let 
$$p_t(x)=\sum_{i=0}^n (-1)^i c_i(t)x^{n-i},$$
with $c_0(t)=1$ and let $\lambda_i(t)$ denote the roots of $p_t(x)$. Recall that
$$ p_t(x)=\E_{S_1,\ldots,S_{k+1}}\left[ \chi(A^{(1)}_{S_1})\ldots\chi(A^{(k)}_{S_k})\cdot \chi((B_t)_{S_{k+1}})\right],$$
where the expectation is over a uniformly random $(k+1)-$wise partition of $[n]$. Since every
matrix that appears above has zero diagonal and consequently zero trace, the second coefficient of each
characteristic polynomial that appears in the above sum is zero, which implies
that $c_1(t)=0$ for all $t\ge 0$, i.e.
\begin{equation} \label{eqn:mean} \sum_{i=1}^n \lambda_i(t)=0\qquad t\ge
0.\end{equation}
 On the other hand, observe that 
\begin{align*}
c_2(t) &= c_2(0)+\P[1\in S_{k+1}] \E \left[\sum_{j\in S_{k+1}\setminus\{1\}}
\det((B_t)_{1j}\bigg|1\in S_{k+1}\right]
\\&= c_2(0)+\P[1\in S_{k+1}]\E\left[\sum_{j=2}^n t\cdot (-|B(1,j)|^2)\cdot \{j \in
S_{k+1}\} \bigg| 1\in S_{k+1}\right]
\\&= c_2(0)-t\cdot\sum_{j=2}^n |B(1,j)^2|\cdot\P[1\in S_{k+1}\land j\in S_{k+1}].
\\&= c_2(0)-\frac{t}{(k+1)^2}\sum_{j=2}^n |B(1,j)|^2.
\end{align*}
Since the first row of $B$ is nonzero this implies that $c_2(t)\rightarrow -\infty$ as
$t\rightarrow\infty$. Combining this with \eqref{eqn:mean}, we have
$$\sum_{i=1}^n \lambda_i^2(t) = c_1(t)^2 -2c_2(t)\rightarrow\infty,$$
so the roots of $p_t$ must be unbounded in $t$. Since the mean of the roots is always zero, we conclude
that $\lambda_{max}(p_t)\rightarrow\infty$ and $\lambda_{min}(p_t)\rightarrow
-\infty$ as $t\rightarrow\infty$.  Thus, $f(t)$ must be monotone increasing in
$t$, and $f(1)\ge f(0)$, as desired.
\end{proof}

Theorem \ref{thm:matrixroot} now follows by a simple inductive argument. 

\begin{proof}
Sequentially applying proposition (\ref{principal}) yields that, 
\[\lambda_{max} \chi[\textbf{A}, B_{([1, k])}] \leq \lambda_{max} \chi[\textbf{A}, B_{([1, k-1])}], \quad k \in [n] .\]
We conclude that, 
\[\lambda_{max} \chi[\textbf{A}, B] \geq \lambda_{max} \chi[\textbf{A}, 0].\]
\end{proof}

To obtain Theorem \ref{thm:shrink} we iterate Theorem \ref{thm:matrixroot} $k-1$
times, yielding
$$ \lambda_{max}\chi[A^{(1)},A^{(2)},\ldots,A^{(k)}]\ge
\lambda_{max}\chi(A^{(1)},\overbrace{0,\ldots,0}^{k-1}).$$
Letting $A=A^{(1)}$ to ease notation, we now observe that
\begin{align*}
\chi[A,\overbrace{0,\ldots,0}^k](x) &= \E_{S_1,\ldots,S_{k}} \chi(A_{S_1})\cdot x^{n-|S_1|}
\\&= \E_{S_1,\ldots,S_{k}} x^{n-|S_1|}\sum_{j=0}^{|S_1|} x^{|S_1|-j}(-1)^j
\sum_{|T|=j} \{T\subseteq S_1\}\det(A_T)
\\&= \sum_{j=0}^n x^{n-j}(-1)^j\sum_{|T|=j}\det(A_T)\P_{S_1,\ldots,S_k}[T\subseteq S_1]
\\&= \sum_{j=0}^n x^{n-j}(-1)^j (1/k)^j\sum_{|T|=j}\det(A_T)
\\&= k^{-n}\sum_{j=0}^n (kx)^{n-j}(-1)^j e_j(A)
\\&=k^{-n}\chi[A](kx),
\end{align*}
whence $\lambda_{max}\chi[A^{(1)}]\le k\lambda_{max} \chi[A^{(1)},\ldots,A^{(k)}],$ as
desired.

\section{Joint Restricted Invertibility}\label{sec:ri}
In this section we prove theorem \ref{thm:ri} by analyzing the expected MDP when we
choose a common submatrix of a $k-$tuple of matrices. We begin by deriving the
a useful formula for this polynomial. Recall that $A_{S}$ denotes the
submatrix of $A$ with columns in $S$ {\em removed}, and $\chi_{S}$ denotes
the correspondingly restricted characteristic polynomial.
\begin{lemma}
Given a $k$-tuple of matrices $\textbf{A}  = (A^{(1)},  \cdots, A^{(k)})$ in
$M_n(\mathbb{C})$, we have that for any $m \leq n$, \[m!\,\sum_{|S| = m}
\chi[\textbf{A}_S ]= \chi^{(m)}[\textbf{A} ].\] \end{lemma}
 \begin{proof}
We use the fact that for any $A \in M_n(\mathbb{C})$, 
\[\operatorname{det}[(Z-A)_{S}] = \partial^{S}\operatorname{det}[Z-A].\]
Using (\ref{restrictedMDP}), we have, 
\begin{eqnarray}\label{chis}
\nonumber  \chi[\textbf{A}_S] &=&  \dfrac{1}{(k!)^n}\partial^{(k-1)[n]} \partial^{S} \left(\prod_{i = 1}^{k} \operatorname{det}[Z-A^{(i)}]\right)\mid_{Z = xI}
\end{eqnarray}
Therefore,
\begin{eqnarray*}
\sum_{|S| = m} \chi[\textbf{A}_S ] &=&\dfrac{1}{(k!)^n}\left[ \sum_{|S| = m} \partial^{S}\right] \partial^{(k-1)[n]}\left(\prod_{i = 1}^{k} \operatorname{det}[Z-A^{(i)}]\right) \mid_{Z = xI},\\
&=& \dfrac{1}{(k!)^n} e_{m}(\partial_1, \cdots, \partial_n)\,\partial^{(k-1)[n]} \left(\prod_{i = 1}^{k} \operatorname{det}[Z-A^{(i)}]\right) \mid_{Z = xI}.
\end{eqnarray*}

The polynomial, 
\[\partial^{(k-1)[n]}\left(\prod_{i = 1}^{k} \operatorname{det}[Z-A^{(i)}]\right),\]
is multiaffine. For any multiaffine polynomial $p(z_1,\cdots, z_n)$ and any $m \leq n$, it is easy to see that, 
\[m!\,\,e_{m}(\partial_1, \cdots, \partial_n)\, p(z_1, \cdots, z_n) \mid_{Z = xI} \,\,= \,\,p^{(m)}(x, \cdots, x).\]

We conclude that, 
\[m!\,\sum_{|S| = m} \chi[\textbf{A}_S ]= \chi^{(m)}[\textbf{A} ].\]
\end{proof}

We now show that the roots of the expected MDP can be used to control the roots
of the best MDP over all $S$, via an interlacing family argument.
\begin{proposition}\label{prop:restrictif}
Let $A^{(1)}, \cdots, A^{(k)}$ be Hermitian matrices in $M_n(\mathbb{C})$. Then,
for any $m \leq n$, there is a subset $S \subset [n]$ of size $m$ such that,
\[\lambda_{max} \,\left(\chi[\textbf{A}_S]\right)\leq \lambda_{max}
\left(\sum_{|S| = m} \chi[\textbf{A}_S]\right) =
\lambda_{max}\left(\chi^{(m)}[\textbf{A} ]\right).\] \end{proposition}
\begin{proof}
 We first claim that the polynomials $\chi[\textbf{A}_{(\{i\})}]$ for $i \in [n]$
have a common interlacer. This is equivalent to verifying that for every
$\alpha_1, \cdots, \alpha_k > 0$, the polynomial 
$$\sum_{i \in [n]} \alpha_i \chi[\textbf{A}_{\{i\}}]$$
 is real rooted. 
We have, 
\[\sum_{i \in [n]} \alpha_i \chi[\textbf{A}_{\{i\}}] = \left(\sum_{i \in [n]} \alpha_i \partial_i\right)
\partial^{(k-1)[n]} \left(\prod_{i = 1}^{k}
\operatorname{det}[Z-A^{(i)}]\right)\mid_{Z = xI}.\] The polynomial
$\partial^{(k-1)[n]} \left(\prod_{i = 1}^{k}
\operatorname{det}[Z-A^{(i)}]\right)$ is real stable, the differential operator
$\left(\sum_{i \in [n]} \alpha_i \partial_i\right)$ preserves real stability and
diagonalization preserves real stability, yielding a univariate real stable
polynomial, which is necessarily real rooted.

Since differentiation preserves interlacing, we conclude that the polynomials
$$\chi^{(j)}[\textbf{A}_{\{i\}}],\quad i\in [n]$$ for fixed $j$ have a common interlacer. 
As a consequence, see \cite{MSSICM}, we have that there is an $i_1$ such that, 
\[\lambda_{max} \left(\chi^{(m-1)}[\textbf{A}_{\{i_1\}}]\right) \leq
\lambda_{max} \left(\sum_{i \in [n]} \chi^{(m-1)}[\textbf{A}_{\{i\}}]\right) = \lambda_{max}\left(\chi^{(m)}[\textbf{A}]\right).\]

Repeating the same argument, we have that there is an $i_2$ such that, 
\[\lambda_{max} \left(\chi^{(m-2)}[\textbf{A}_{\{i_1, i_2\}}]\right) \leq
\lambda_{max} \left(\sum_{i \in [n] \setminus i_1} \chi^{(m-2)}[\textbf{A}_{\{i_1,
i\}}]\right) = \lambda_{max}\left(\chi^{(m-1)}[\textbf{A}_{\{i_1\}}]\right) \leq \lambda_{max}\left(\chi^{(m)}[\textbf{A}]\right).\]
Iterating this, we see that there is a subset $S$ of size $m$ such that, 
\[\lambda_{max}\left(\chi[\textbf{A}_S]\right) \leq \lambda_{max}\left(\chi^{(m)}[\textbf{A}]\right).\]
\end{proof}

To finish the proof, we appeal to the following Theorem from \cite{MRI}[Lemmas 4.3, 4.5], which shows
how the roots of a univariate polynomial shrink under taking many derivatives.
\begin{theorem}[Root Shrinking]\label{thm:mohangauss2}
For any real rooted polynomial $p$ of degree $n$ with roots in $[-1, 1]$, with average of roots $0$ and with average of the squares of the roots $\alpha$ and any $c \leq (1+\alpha)^{-1}$, 
\[\lambda_{max} p^{(n(1-c))} \leq c(1-\alpha) + 2\sqrt{c(1-c)\alpha}.\]
\end{theorem}

In the proof that follows, given a tuple of Hermitians $\textbf{A} = (A^{(1)}, \cdots, A^{(k)})$, we will use the notation, 
\[\operatorname{Tr}(\textbf{A}^2) := \sum_{i \in [k]}
\operatorname{Tr}((A^{(i)})^2)\]

\begin{proof}[Proof of Theorem \ref{thm:ri}]
We start off by noting that the characteristic polynomial of a zero diagonal $n \times n$ Hermitian matrix $A$ is of the form,
\[x^{n} - \dfrac{\operatorname{Tr}(A^2)}{2}x^{n-2} + \text{ lower order terms}.\]

Observe that for any $k$ tuple of $n \times n$ zero diagonal Hermitians $\bA :=A^{(1)},\ldots,A^{(k)}$ and partition $S_1\cup\ldots\cup S_k=[n]$, each of the matrices $A^{(i)}_{S_i}$ has
eigenvalues in $[-1,1]$. We may write the product of the characteristic polynomials as
\[ \chi[A^{(1)}](x)\ldots \chi[A^{(k)}](x) = x^{n} - \dfrac{\operatorname{Tr}(\textbf{A}^2)}{2} x^{n-2} + \text{ lower order terms},\]
and it is evident that this polynomial positive on $(1,\infty)$ (since it is
monic) and has sign $(-1)^n$ on $(-\infty,1)$. Taking an average over all
partitions and noting that each pair of indices has probability $1/k^2$ of being
in any particular $S_i$, we find that
$$q(x):=\chi[A^{(1)},\ldots,A^{(k)}](x)=\E_{S_1,\ldots, S_k} \prod_{i=1}^k
\chi[A^{(i)}_{S_i}](x) =  x^{n} - \dfrac{\operatorname{Tr}(\textbf{A}^2)}{2k^2} x^{n-2} + \cdots$$
also has these properties, so all of its roots must lie in $[-1,1]$. Moreover, by above, the sum of the roots of $q$ is zero and the average of the squares of the roots is 
\[\alpha = \dfrac{\operatorname{Tr}(\textbf{A}^2)}{nk^2} \leq \frac{nk}{nk^2}\le \dfrac{1}{k},\]
since $\|A^{(i)}\|\le 1$ for every $j$.

Writing, $m=(1-c)n$ for $c$ sufficiently small, Theorem
\ref{thm:mohangauss2} yields that,  
\[\lambda_{max} \left(\chi^{(m)}[\bA]\right)\leq c(1-\alpha) + 2\sqrt{c(1-c)\alpha} \leq \dfrac{c(k-1)}{k} + 2\sqrt{\dfrac{c(1-c)}{k}}.\]
Applying Theorem \ref{thm:shrink}, this implies that
\[\lambda_{max}(A^{(i)}_\sigma)\le c(k-1)+2\sqrt{c(1-c)k}, \quad i \in [k].\]
Given $\epsilon < 1$, setting $c = \epsilon^2/6k$, we see that, 
\[ c(k-1)+2\sqrt{c(1-c)k} \leq ck + 2\sqrt{ck} = \dfrac{\epsilon^2}{6} + \dfrac{2\epsilon}{\sqrt{6}} < \epsilon,\]
as desired.
\end{proof}

\section{Joint Paving}\label{sec:paving}
In this section we prove Theorem \ref{thm:paving}. As with restricted
invertibility, the idea is to construct an interlacing family on certain MDPs,
this time indexed by $r-$wise partitions of $[n]$. To this end, let $\cP_r$
denote the set of partitions of $[n]$ into $r$ subsets. Given a partition
$\textbf{S}=(S_1,\ldots,S_r)\in\cP_r$ and a tuple of matrices
$\bA=(A^{(1)},\ldots,A^{(k)})$, recall that we defined,
$$ \chi[\bA(\textbf{S})] := \prod_{j=1}^r \chi[\bA(S_j)].$$

We will show that these polynomials can be used to construct an interlacing tree, as follows. For any $m \leq n$ and any ordered $r$ partition (that may include empty sets) of $[m]$, $\textbf{T} = (T_1, \cdots, T_r)$, i.e. $T_1 \amalg \cdots \amalg T_r = [m]$, define,
\begin{eqnarray}\label{qdef} 
q(\textbf{T}) := r^{m-n}\sum_{\textbf{T} \prec \textbf{S} \in \cP_r} \chi[\textbf{A}(\textbf{S})].
\end{eqnarray}

The ordering $\textbf{T} \prec \textbf{S}$ defined on $r$ element subsets of $[n]$ means that $T_i \subset S_i$ for $i \in [r]$. In other words, $q(\textbf{T})$ is the expected characteristic polynomial over all extensions of the $r$ partition $\textbf{T}$ of $[m]$ to $r$ partitions of $[n]$. These polynomials have concise combinatorial expressions, extending the formula in (\ref{MDPPaving}), which applies to the case when $\textbf{T}$ is a partition of $[n]$. 

\begin{proposition}
Let $\textbf{A} = (A^{(1)}, \cdots, A^{(k)})$ be a $k$ tuple of matrices in $M_n(\mathbb{C})$ and let $m \leq n$ and let $\textbf{T} = (T_1, \cdots, T_r)$ be a $r$ partition of $[m]$. Then, 
 \begin{eqnarray}
q(\textbf{T}) = \dfrac{1}{(k!)^{rm}(kr!)^{n-m}}\prod_{j \in [m+1,n]}\left(\sum_{ i \in [r]}\partial_{(i)}^{j}\right)^{kr-1}\left(\prod_{i \in [r]} \partial_{(i)}^{k[m]\setminus T_i}\right) \prod_{i = 1}^{r} \prod_{j = 1}^{k}\operatorname{det}[Z_i - A^{(j)}]\mid_{Z_1 = \cdots = Z_r = xI}
\end{eqnarray}
\end{proposition}

\begin{proof}
We have,
\begin{eqnarray*}
q(\textbf{T}) &=& r^{m-n}\sum_{\textbf{T} \prec \textbf{S} \in \cP_r} \chi[\textbf{A}(\textbf{S})]\\
&=& \dfrac{r^{m-n}}{(k!)^{rn}}\sum_{\textbf{R}\in \mathcal{P}_r([m+1,n])}\left(\prod_{i \in [r]} \partial_{(i)}^{k[n]\setminus T_i \amalg R_i}\right) \prod_{i = 1}^{r} \prod_{j = 1}^{k}\operatorname{det}[Z_i - A^{(j)}]\mid_{Z_1 = \cdots = Z_r = xI}\\
&=& \dfrac{r^{m-n}}{(k!)^{rn}}\left(\sum_{\textbf{R}\in \mathcal{P}_r([m+1,n])}\prod_{i \in [r]} \partial_{(i)}^{k[m+1,n]\setminus R_i}\right) \left(\prod_{i \in [r]} \partial_{(i)}^{k[m]\setminus T_i}\right) \prod_{i = 1}^{r} \prod_{j = 1}^{k}\operatorname{det}[Z_i - A^{(j)}]\mid_{Z_1 = \cdots = Z_r = xI}\\
&=& \dfrac{1}{(k!)^{rm}(kr!)^{n-m}}\prod_{j \in [m+1,n]}\left(\sum_{ i \in [r]}\partial_{(i)}^{j}\right)^{kr-1}\left(\prod_{i \in [r]} \partial_{(i)}^{k[m]\setminus T_i}\right) \prod_{i = 1}^{r} \prod_{j = 1}^{k}\operatorname{det}[Z_i - A^{(j)}]\mid_{Z_1 = \cdots = Z_r = xI}
\end{eqnarray*}
\end{proof}

We now show that the polynomials $q(\textbf{T})$ form an interlacing family. We will use the notation $e_j^i$ for $i \in [n]$ to denote a collection of $r$ subsets of $[n]$, with the $i$'th subset equalling the singleton $\{j\}$ and the other subsets being empty. Formally,
\[e_j^i := (\phi, \cdots, \phi, \overbrace{\{j\}}^{i}, \phi, \cdots, \phi).\]

\begin{lemma}\label{MDPInterlacing}
Let $\textbf{A} = (A^{(1)}, \cdots, A^{(k)})$ be a $k$ tuple of Hermitian matrices in $M_n(\mathbb{C})$ and let $m \leq n$ and let $\textbf{T} = (T_1, \cdots, T_r)$ be a $r$ partition of $[m]$. Then, 
\begin{enumerate}
 \item The polynomials $q(\textbf{T} \cup e_{m+1}^{i})$ for $i \in [r]$ have a common interlacer. 
 \item The expected characteristic polynomial satisfies, 
 \[\mathbb{E}_{i \in [r]} q(\textbf{T} \cup e_{m+1}^{i}) = q(\textbf{T}).\]
 \item Consequently, there is a $i \in [r]$ such that, 
 \[\lambda_{max} \left(q(\textbf{T} \cup e_{m+1}^{i})\right) \leq \lambda_{max} \left(q(\textbf{T})\right).\]
\end{enumerate}
\end{lemma}
\begin{proof}
 Let $p$ be the polynomial, 
 \[p = \dfrac{1}{(k!)^{r(m+1)}(kr!)^{n-m-1}}\prod_{j \in [m+2,n]}\left(\sum_{ i \in [r]}\partial_{(i)}^{j}\right)^{kr-1}\left(\prod_{i \in [r]} \partial_{(i)}^{k[m]\setminus T_i}\right) \prod_{i = 1}^{r} \prod_{j = 1}^{k}\operatorname{det}[Z_i - A^{(j)}].\]
 We have that, 
 \begin{eqnarray*}
  q(\textbf{T} \cup e_{m+1}^{i}) &=& \partial_{(i)}^{(k-1)\{m+1\}} \left(\prod_{\substack{j \neq i\\ j \in [r]}} \partial_{(i)}^{k\{m+1\}}\right)\,p \mid_{Z_1 = \cdots = Z_r = xI} 
  \end{eqnarray*}
  To show that the polynomials $q(\textbf{T} \cup e_{m+1}^{i})$ have a common interlacer, we need to show that for any non-negative real numbers $\{\alpha_i\}_{i \in [r]}$, the polynomial $\sum_{i \in [r]} \alpha_i q(\textbf{T} \cup e_{m+1}^{i})$ is real rooted. We see that, 
  \begin{eqnarray}\label{convexcombination}
   \sum_{i \in [r]} \alpha_i q(\textbf{T} \cup e_{m+1}^{i}) &=& D \left(\prod_{ j \in [r]} \partial_{(i)}^{(k-1)\{m+1\}}\right)\,p \mid_{Z_1 = \cdots = Z_r = xI}, 
  \end{eqnarray}
  where $D$ is the differential operator, 
  \[D = \sum_{i \in [r]}\alpha_i \prod_{\substack{j \neq i\\ j \in [r]}} \partial_{(j)}^{\{m+1\}}.
  \]
  The symbol of this diffrential operator is $F_D(w):=\sum_{i \in [r]}\alpha_i
\prod_{\substack{j \neq i\\ j \in [r]}} w_j$, and $F_D(-w)$ may be written as,
  \[(-1)^{r-1}\sum_{i \in [r]}\alpha_i \prod_{\substack{j \neq i\\ j \in [r]}}
w_j = (-1)^{r-1}\left(\sum_{i \in [r]} \alpha_i \partial_i\right) \prod_{i \in
[r]} w_i.\]
  We conclude that $F_D(-w)$ is real stable, which by \ref{thm:bbweyl} shows that $D$ preserves real stability. We conclude that the polynomial in (\ref{convexcombination}) is real rooted. This proves the first assertion of the proposition. 
  
  For the second, we see that, 
  \begin{eqnarray*}
   \mathbb{E}_{i \in [r]} q(\textbf{T} \cup e_{m+1}^{i}) &=& \dfrac{1}{r}\sum_{i \in [r]} \partial_{(i)}^{(k-1)\{m+1\}} \left(\prod_{\substack{j \neq i\\ j \in [r]}} \partial_{(i)}^{k\{m+1\}}\right)\,p \mid_{Z_1 = \cdots = Z_r = xI} \\
   &=&\dfrac{1}{r} \dfrac{(k!)^{r-1}(k-1)!}{(kr-1)!}\left(\sum_{ i \in [r]}\partial_{(i)}^{\{m+1\}}\right)^{kr-1} p \mid_{Z_1 = \cdots = Z_r = xI}\\
   &=&  \dfrac{(k!)^{r}}{(kr)!}\left(\sum_{ i \in [r]}\partial_{(i)}^{\{m+1\}}\right)^{kr-1} p \mid_{Z_1 = \cdots = Z_r = xI}\\
   &=& \dfrac{1}{(k!)^{rm}(kr!)^{n-m}}\prod_{j \in [m+1,n]}\left(\sum_{ i \in [r]}\partial_{(i)}^{j}\right)^{kr-1}\\
   &&\left(\prod_{i \in [r]} \partial_{(i)}^{k[m]\setminus T_i}\right) \prod_{i = 1}^{r} \prod_{j = 1}^{k}\operatorname{det}[Z_i - A^{(j)}]\mid_{Z_1 = \cdots = Z_r = xI}\\
   &=& q(\textbf{T}).
  \end{eqnarray*}
The third assertion follows from Lemma \ref{lem:interlacing}.

\end{proof}

We now give a succinct formula the expected characteristic polynomial over all partitions. We do this by deriving a formula for the expectation of this polynomial over a random partition, which turns out to be an MDP.
\begin{lemma}\label{ExpectedMDP}
Given a $k$-tuple of matrices $\textbf{A} = (A^{(1)}, \cdots, A^{(k)})$ in
$M_n(\mathbb{C})$, we have, \[\mathbb{E}_{\textbf{S} \in \mathcal{P}_r}
\chi[\textbf{A}(\textbf{S})]   = \chi[\overbrace{\textbf{A},  \cdots,
\textbf{A}}^{r}].\]
\end{lemma}
\begin{proof} By the definition of an MDP, the right hand side is a uniform
average over all $rk-$partitions of
$[n]$. Consider the following two stage sampling process: 
\begin{enumerate}
\item Choose an $r-$partition $\X\in\cP_r$.
\item For each $X_j\in\X$, choose a subpartition
$X_j=S_j^{(1)}\cup\ldots S_j^{(k)}$ into $k$ parts.
\end{enumerate}
We claim that the  partition $\{S^{(i)}_j\}_{j\le r, i\le k}$ output by the
procedure is uniformly random. To see why, observe that the process constructs a
random allocation $[n]\rightarrow [r]\times [k]$ of elements to subsets by
choosing the first coordinate uniformly and then the second, which certainly
generates a uniformly random allocation. Labeling the $r$ copies of $A^{(i)}$ as
$A^{(i,j)}$ with $j=1,\ldots,r$, we have:
\begin{align*}
\chi[\overbrace{\bA,\ldots,\bA}^{r}] 
&= \E_{\{S^{(i)}_j\}_{ij}\sim\cP_{rk}} \prod_{j=1}^r\prod_{i=1}^k
\chi[A^{(i,j)}(S^{(i)}_j)]
\\&= \E_{\X\sim\cP_r} \left[\E_{\{S^{(1)}_j,\ldots,S^{(k)}_j\in\cP(X_j)\}_j} 
\prod_{j=1}^r\prod_{i=1}^k \chi[A^{(i,j)}(S^{(i)}_j)] \bigg| \X\right]
\\&= \E_{\X\sim\cP_r} \left[ \prod_{j=1}^r\E_{\{S^{(1)}_j,\ldots,S^{(k)}_j\in\cP(X_j)\}_j} 
\prod_{i=1}^k \chi[A^{(i,j)}(S^{(i)}_j)] \bigg| \X\right]
\\&\qquad\textrm{since $(S^{(1)}_{j},\ldots,S^{(k)}_j)$ are independent
for distinct $j$, conditioned on $\X$}
\\&= \E_{\X\sim\cP_r} \left[ \prod_{j=1}^r \chi_{X_j}[A^{(1)},\ldots,A^{(k)}]\bigg| \X\right]
\\&= \E_{\X\sim\cP_r} \chi^\X[\bA],
\end{align*}
as desired.
\end{proof}

\begin{theorem}
 Let $\textbf{A} = (A^{(1)}, \cdots, A^{(k)})$ be a $k$ tuple of Hermitian matrices in $M_n(\mathbb{C})$. Then, there is a partition $\textbf{S} = (S_1, \cdots, S_r)$ of $[n]$ such that, 
 \[\lambda_{max} \left(\chi[\textbf{A}(\textbf{S})]\right) \leq \lambda_{max} \left(\mathbb{E}_{\textbf{S} \in \mathcal{P}_r} \chi[\textbf{A}(\textbf{S})]\right) =  \lambda_{max} \left(\chi[\overbrace{\textbf{A}, \cdots, \textbf{A}}^r]\right).\]
\end{theorem}
\begin{proof}
 Iterating the third statement in proposition (\ref{MDPInterlacing}), we see that there is a partition $\textbf{S} = (S_1, \cdots, S_r)$ of $[n]$ such that, 
 \[\lambda_{max}\left(q(\textbf{S})\right) \leq \lambda_{max}\left(q(\overbrace{\phi, \cdots, \phi}^r)\right).\]
 From the definition (\ref{qdef}), it is clear that, 
 \[q(\textbf{S}) = \chi[\textbf{A}(\textbf{S})].
 \]
 Also, by lemma (\ref{ExpectedMDP}), we have that, 
 \begin{eqnarray*}
  q(\overbrace{\phi, \cdots, \phi}^r) := r^{-n}\sum_{ \textbf{S} \in \cP_r} \chi[\textbf{A}(\textbf{S})] = \chi[\overbrace{\textbf{A}, \cdots, \textbf{A}}^r]
 \end{eqnarray*}
This completes the proof. 
\end{proof}

Finally, we prove a bound on the largest root of the expected characteristic
polynomial above. This follows by observing that the MDP of a tuple of matrices can be written as the mixed
characteristic polynomial of related matrices, which immediately allows us to transfer known
bounds for roots of mixed characteristic polynomials to MDPs.
\begin{lemma}\label{lem:mdpmcp}
Let $B^{(1)}, \cdots, B^{(k)}$ be positive semidefinite matrices in
$M_n(\mathbb{C})$ and suppose $\{v_j^i\}_{i\le k, j\le n}$ are vectors in
$\mathbb{C}^n$
such that 
$$ B^{(i)}(j_1,j_2)=\langle v^i_{j_1},v^i_{j_2}\rangle,$$
i.e., $B^{(i)}$ is the Gram matrix of the $v^i_j$.
Letting, 
\[X_i = \bigoplus_{j = 1}^{k} v^{i}_j (v^{i}_j)^{*}, \quad i \in [n],\]
we have that,
\[x^{(k-1)n}\chi[kB^{(1)}, \cdots, kB^{(k)}] = \mu[X_1, \cdots, X_m].\]

\end{lemma}
\begin{proof}
Define random vectors $r_1, \cdots, r_n$ by letting $r_j$ be the random vector taking values in $\overbrace{\mathbb{C}^n \oplus \cdots \oplus \mathbb{C}^n}^k$ that takes values $0 \oplus \cdots \oplus \overbrace{\sqrt{k}v^{i}_{j}}^j \oplus \cdots \oplus 0$ for $i \in [k]$ with probability $1/k$ apiece. We note that, 
\[\mathbb{E}\left[r_jr_j^{*}\right] = X_j, \quad j \in [n].\]
By the definition of the mixed characteristic polynomial, we have that, 
\[\mu[X_1, \cdots, X_m] = \dfrac{1}{k^{n}} \sum_{S_1 \amalg \cdots \amalg S_k = [n]} \prod_{i \in [k]}\chi\left[k\sum_{j \in S_i} v^{i}_j (v^{i}_{j})^{*}\right]. \] 
We note that, 
\[\chi\left[k\sum_{j \in S_i} v^{i}_j (v^{i}_{j})^{*}\right] = x^{n - |S_i|} \chi\left[k\,B^{(i)}_{S_i}\right].\]
We conclude that, 
\[\mu[X_1, \cdots, X_m]  = \dfrac{1}{k^{n}} x^{(k-1)n}  \sum_{S_1 \amalg \cdots \amalg S_k = [n]} \prod_{i = 1}^{k} \chi[k\,B^{(i)}(S_i)].\]
The expression on the right equals $x^{(k-1)n}\chi[k\,B^{(1)}, \cdots, k\,B^{(k)}] $, by definition. 
\end{proof}

\vspace{.1in}

\noindent {\it Remark.} Lemma \ref{lem:mdpmcp} shows that the MDP, which is an
expectation over uniform $k-$partitions of $[n]$, can be written as a mixed
characteristic polynomial in the covariance matrices $X_i$ which encode the
(uniform over $[k]$) distribution of each coordinate in $[n]$. With minor
modifications the lemma can be generalized to the setting of distributions over
{\em nonuniform} partitions, in which each coordinate $i\in [n]$ is assigned to an
element of $[k]$ independently according to some distribution $\mu_i$. With this
setup, the operation of taking a mixture of two such distributions
$\mu:=\mu_1\otimes\ldots\otimes\mu_n$ and $\nu:=\nu_1\otimes\ldots\otimes\nu_n$
simply corresponds to taking an average of the corresponding covariance matrices
$X_i$, yielding another mixed characteristic polynomial, which is necessarily
real-rooted by Theorem \ref{thm:if2}.
This fact can be used to give an alternate proof of Lemma
\ref{MDPInterlacing}, since the relevant averages of the conditional expected polynomials
which appear during the interlacing argument are simply mixtures of such
distributions. We have chosen to give a direct proof to keep the presentation
self-contained, but would like to mention that conceptually the interlacing
argument for MDPs works for the same reason as it does for mixed characteristic
polynomials.

The above identity allows us to use the estimates of Theorem \ref{thm:if2} on roots of mixed characteristic polynomials to get effective upper bounds for mixed determinantal polynomials. 

\begin{theorem}\label{thm:rootbound}
Let $A^{(1)}, \cdots, A^{(k)}$ be a $k$ tuple of zero diagonal Hermitian contractions, where $k \geq 2$. Then, 
\[\lambda_{max}\chi[A^{(1)}, \cdots, A^{(k)}] < \dfrac{3\sqrt{2}}{\sqrt{k}}.\]
\end{theorem}

\begin{proof}
Define the matrices, 
\[B^{(i)} = \dfrac{I + A^{(i)}}{2}, \quad i \in [k].\]
We note that, 
\[\sigma\left(\chi\left[A^{(1)}, \cdots, A^{(k)}\right]\right) = 2\sigma\left(\chi[B^{(1)}, \cdots, B^{(k)}]\right) - 1 .\]
Letting the $X_i$ be as in the proof of the previous lemma, we note that, 
\[\sum_{j \in [n]} X_j =B^{(1)} \oplus \cdots \oplus B^{(k)},\]
yielding that $X_1 + \cdots + X_n$ is a positive contraction. Further, given that all the vectors $v^{i}_{j}$ (in the proof of the previous lemma) have squared lengths $1/2$, we note that,
\[\operatorname{Trace}\left(X_i\right) = \dfrac{k}{2}, \quad i \in [n].\]
Theorem \ref{thm:if2} now yields that, 
\[\lambda_{max}\left( \chi[B^{(1)}, \cdots, B^{(n)}] \right) = \dfrac{1}{k}\lambda_{max} \left(\mu[X_1, \cdots, X_n]\right) <  \dfrac{(1+\sqrt{k/2})^2}{k} = \dfrac{1}{k} + \dfrac{\sqrt{2}}{\sqrt{k}} + \dfrac{1}{2}.\]
This in turn yields that, 
\[\lambda_{max}\chi[A^{(1)}, \cdots, A^{(k)}] <  \dfrac{2\sqrt{2}}{\sqrt{k}} + \dfrac{2}{k} \leq \dfrac{3\sqrt{2}}{\sqrt{k}}, \quad \text{ if } \, \, k \geq 2.\]
\end{proof}

\begin{proof}[Proof of Theorem \ref{thm:paving}] Given any zero diagonal
$n\times n$ Hermitian $\bA=(A^{(1)},\ldots,A^{(k)})$ and $r$, Lemma
\ref{MDPInterlacing} tells
us that there exists a partition $\X\in\cP_r$ such that
$$\lambda_{max}\chi^\X[\bA]\le
\lambda_{max}\chi[\overbrace{\bA,\ldots,\bA}^{r}].$$
By Theorem \ref{thm:rootbound}, the right hand side is at most
$3\sqrt{2}/\sqrt{rk}$. By Theorem \ref{thm:shrink} this means that for every
$i=1,\ldots,k$ we have
$$\lambda_{max}\chi^\X[A^{(i)}]< k\cdot\frac{3\sqrt{2}}{\sqrt{rk}} =
3\sqrt{2}\sqrt{\frac{k}{r}}.$$
In order to make the latter quantity at most $\epsilon$ it is sufficient to take
$r=18k/\epsilon^2$, as desired.
\end{proof}

We now state a corollary that gives paving estimates for non-Hermitian matrices. We also indicate a simple trick that allows us to get blocks that are all uniformly bounded. 

\begin{corollary}\label{lem:equalpaving} If $A$ is a zero diagonal contraction, and $r$ is an integer, then there is a two-sided paving of $A$ with $r$ such that each block has norm at most $12\sqrt{2}/\sqrt{r}$. Further, there is a paving into $2r$ blocks such that each block has size at most $\lfloor m/r \rfloor$ and norm at
	most $12\sqrt{2}/\sqrt{r}$.
\end{corollary} 
\begin{proof}
By applying Theorem \ref{thm:paving} to the tuple $\textbf{S} =(\operatorname{Re}(A),-\operatorname{Re}(A),\operatorname{Im}(A),-\operatorname{Im}(A))$, we may obtain a paving
	$X_1,\ldots,X_r$ such that,
	\[\lambda_{max} (P_{X_i} S P_{X_i}) \leq \dfrac{\sqrt{72}}{\sqrt{r}}, \quad S \in \textbf{S}, \, i \in [r].\]
	This implies that,
		\[||P_{X_i} \operatorname{Re}(A) P_{X_i}||,\, ||P_{X_i} \operatorname{Im}(A) P_{X_i}|| \leq \dfrac{\sqrt{72}}{\sqrt{r}} = \dfrac{6\sqrt{2}}{\sqrt{r}},\quad  i \in [r].\]
		This in turn implies that,
		\[||P_{X_i} A P_{X_i}|| \leq ||P_{X_i} \operatorname{Re}(A) P_{X_i}|| + ||P_{X_i} \operatorname{Im}(A) P_{X_i}|| \leq 2\cdot \dfrac{6\sqrt{2}}{\sqrt{r}} = \dfrac{12\sqrt{2}}{\sqrt{r}}, \quad i \in [r].\]
	
For the second assertion, take a paving of $A$ into $r$ blocks $\{X_1,
	\cdots, X_{r}\}$ such that each block has norm at most
	$12\sqrt{2}/\sqrt{r}$. Subpartition each block $X_i$ which has more
	than $\lfloor m/r \rfloor$ elements into sub blocks, all of which save
	at most one have size at least $\lfloor m/r \rfloor$ and with at most one sub
	block of smaller size. Suppose we get $s$ blocks of size $\lfloor m/r
	\rfloor$ and $t$ blocks of smaller size (note that $t \leq r$. We have
	that $s\lfloor m/r \rfloor + t \leq m$, yielding that the total number
	of blocks $s+t$ is at most $2r$. Finally since passing to a sub-block
	cannot increase the norm, we get the required norm estimate.
\end{proof}

\section{A Simplified Proof of the Quantitative Commutator Theorem}
In \cite{JOS}, Johnson, Ozawa and Schechtman proved the following theorem,

\begin{theorem}[Johnson, Ozawa, Schechtman, 2013]\label{thm:jos}
For any $\epsilon > 0$, there is a constant $K(\epsilon)$ such that every zero trace matrix $A \in M_m(\mathbb{C})$ may be written as $A = [B, C]$, such that $||B||\, ||C|| \leq K(\epsilon)m^{\epsilon}\, ||A||$.
\end{theorem}
They also remarked that optimizing their proof yields a bound of $e^{O(\sqrt{\log m\log \log m})}$.

In this section, we explain how to use Theorem \ref{thm:paving} to give a simple
proof of a  slight improvement of the main theorem of \cite{JOS}. The main idea is
the same as that paper, and consists of three steps. Given $A \in M_m(\mathbb{C})$, the matrix $B$ will be chosen to be diagonal and with entries in the square $\S:=[-1,1]\times [-i,i]\subset \mathbb{C}$.
\begin{enumerate}
\item Given an $A \in M_m(\mathbb{C})$, write it as an
$r\times r$ block matrix containing $r^2$ significantly smaller matrices such
that the diagonal blocks $A_{ii}$, $i=1,\ldots,r$ are square and have smaller norm.
\item Recursively solve the commutator problem for these blocks, to obtain $B_{ii}$
and $C_{ii}$ such that $A_{ii}=[B_{ii},C_{ii}]$ and every $B_{ii}$ is diagonal
with all eigenvalues in the square $\S:=[-1,1]\times [-i,i]\subset \mathbb{C}$.
\item Show that for the same $B$, the off-diagonal blocks $C_{ij}$ may be chosen to have small norm. This
is done by appealing to an explicit solution of the equation
$A_{ij}=B_{ii}C_{ij}-C_{ij}B_{jj}$ due to Rosenblum \cite{rosenblum}, which
implies that the $C_{ij}$ have small norm whenever the spectra of 
$B_{ii},B_{jj}$ are separated, which is achieved by appropriately embedding
these spectra in a tiling of $\S$ by $r$ smaller well-separated squares.
\end{enumerate}

Following \cite{JOS}, let $\lambda(A)$ denote the smallest norm of a matrix
$C$ such that $A=[B,C]$ with the spectrum of $B$ contained in $\S$.  Then we have
the following, which is closely related to Claim $1$ from \cite{JOS}.
\begin{lemma}\label{lem:claim1}
	Suppose $A$ is an $m\times m$ zero diagonal contraction and $X_1,\ldots,X_{r}$ is a partition
	of $[m]$, with $r^2$ a perfect square.  Then,
	$$\lambda(A)\le 2\sqrt{r}\max_{i\le r}\lambda(A_{ii})
	+ 6r^{3/2},$$
	where $A_{ij}$ is the submatrix $A(X_i,X_j)$. 
\end{lemma}
\begin{proof} The proof is identical to the proof of Claim $1$ from
	\cite{JOS}, with minor adjustments. Let $B_{ii},C_{ii}$ be matrices
	with $A_{ii}=[B_{ii},C_{ii}]$, for $i=1,\ldots,r$ such that the
	spectra of the $B_{ii}$ lie in $\S$ and $\|C_{ii}\|\le\lambda(A_{ii})$. Subdivide $\S$ into
	$r=\sqrt{r} \times \sqrt{r}$ smaller squares $\S_1,\ldots,\S_{r}$ of side length $2/\sqrt{r}$,
	and let $B_{ii}'$ be $B_{ii}$ scaled by $1/(2\sqrt{r})$ and
	translated by a multiple of the identity so that its spectrum lies
	inside $\S_i$ at distance at least $1/(2\sqrt{r})$ from $\partial \S_i$.
	Since translation does not change the commutator, we still have
	$A_{ii}=[B_{ii}',C_{ii}']$ where
	$C_{ii}'=2\sqrt{r}\,C_{ii}$. Combining these facts, we
	get
	$$\|\bigoplus_{i\le r} C_{ii}'\|\le
	2\sqrt{r}\max_{i\le r}\lambda(A_{ii}).$$

	We now handle the off-diagonal blocks. A result of Rosenblum
	\cite{rosenblum} asserts that for any square matrices $A,S,T$ of the same dimension, the matrix
	$$ X:=\frac{1}{2\pi i}\int_\gamma (z-S)^{-1}A(z-T)^{-1}dz$$
	satisfies $A=SX-XT$, where $\gamma$ is a positively oriented simple closed curve containing the spectrum of $S$ and 
	excluding the spectrum of $T$. One can easily verify by the residue theorem that this solution also works when 
	$A$ is rectangular and $S,T$ are square but of possibly different
	dimensions. Apply this with $A=A_{ij},\, S=B_{ii}', \,T=B_{jj}'$, and $\gamma=\partial
	\S_i$ to obtain a solution $C_{ij}$. Note that
	$\|(z-S)^{-1}\|,\|(z-T)^{-1}\|\leq 2\sqrt{r}$ on $\gamma$, and the length of $\gamma$ is at most $8/\sqrt{r}$ 
	so by the triangle inequality:
	$$ \|C_{ij}\|\le \dfrac{32\sqrt{r}}{2\pi}\|A_{ij}\| \le 6\sqrt{r}\|A\| \le 6\sqrt{r}.$$

	Since the norm of $C$ is at most the sum of the norms of its cyclic diagonals
	(i.e., submatrices containing blocks $C_{ij}$ for $i-j=\ell \, (\operatorname{mod} m)$ for
	$\ell=1, \cdots, (r^2-1)$), and the norm of each diagonal is the maximum
	over the blocks it contains, we conclude that
	$$\|C\|\le 2\sqrt{r}\max_{i\le r}\lambda(A_{ii}) +
	6r^{3/2},$$
	as desired.
\end{proof}	

We now  apply this recursively, together with a use of corollary (\ref{lem:equalpaving}),
	
\begin{theorem}  If $A$ is an $m\times m$ zero trace complex matrix then there exist $B,C$
	with $A = [B,C]$ and with $\|B\|\|C\|\le 300\, e^{9\sqrt{\log(m)}}||A||.$
\end{theorem}
\begin{proof}
	We may assume that $A$ has norm one. Let $\Lambda(m)$ denote the maximum of $\lambda(A)$ over all $m\times m$ zero trace matrices with $\|A\|\le 1$, and note that $\lambda(m)$ is nondecreasing in $m$. Let $r$ be such 
	that $2r$ is a perfect square. Apply Lemma \ref{lem:equalpaving} to $A$ to obtain $2r$ 
	blocks $A_{ii}$ of size at most $\lfloor m/r\rfloor $ such that each block has norm at
	most $12\sqrt{2}/\sqrt{r}$. Since $\lambda(A_{ii})\le
	\left(17/\sqrt{r}\right)\lambda(\lfloor m/r \rfloor)$, applying Lemma \ref{lem:claim1} 
	yields the recurrence:
	\begin{align*}
		\lambda(m) &\le
	2\sqrt{2r}\cdot\frac{17}{\sqrt{r}}\lambda(\lfloor m/r\rfloor )+16\sqrt{2}r^{3/2}\\
	&\leq 50\,\lambda(\lfloor m/r \rfloor)+25r^{3/2}
		\\&\le 25\,r^{3/2}\cdot (1+50^1+50^2+\ldots+50^{\log(m)/\log(r)})
		\\&\le 25\, r^{3/2}\cdot e^{4\log(m)/\log(r)}
	\end{align*}
	
	This expression is minimized when $r = e^{\sqrt{8/3}\sqrt{\operatorname{log}(m)}}$. To finish the proof, we remove the constraint on $r$. Replace $r$ by the next even perfect square, which is at most $4r$ (since the ratio between consecutive even squares is 
	$(2r+2)^2/4r^2 = 1 + 2/r + 1/r^2 \leq 4$). The
	final bound is at most the expression evaluated at $r = 4e^{\sqrt{8/3}\sqrt{\operatorname{log}(m)}}$, where it equals,
	\[ 25\cdot 8\cdot e^{\sqrt{6\operatorname{log}(m)}}e^{\sqrt{6}\operatorname{log}(m)/(\operatorname{log}(4) +\sqrt{\operatorname{log}(m)})} \leq 200 \,e^{2\sqrt{6\operatorname{log}(m)}}\leq 200 \,e^{9\sqrt{\operatorname{log}(m)}}.\]

	This yields a $B$ with $\|B\|\le\sqrt{2}$ and $\|C\|\le
	200\,e^{9\sqrt{\log(m)}}$, as desired.  \end{proof}

We have not made any attempt to optimize the constants in the above proof.

\section{Concluding remarks}

We end this paper with three remarks; we first show how our multi-paving bounds are asymptotically sharp, both in the number of matrices $k$ and the desired norm $\epsilon$. We next make a calculation with characteristic polynomials of signed adjacency matrices of graphs to get a weak result concerning eigenvalue bounds for adjacency matrices of non-bipartite regular graphs. Thirdly, we discuss a vectorial version of Anderson paving.

\subsection{Tightness of bounds}

The bounds in the main theorem on multi-paving are sharp. For any $k \in
\mathbb{N}$ and $\epsilon < 1$, one has examples of $k$ zero diagonal matrices
that need at least $r = k\lfloor \epsilon^{-2} \rfloor $ blocks to be jointly
$(r, \epsilon)$ paved. 

Fix a positive integer $k$ and $\epsilon < 1$ and let  $m = \lfloor \epsilon^{-2} \rfloor$.  Let $U$ be the cyclic left shift on $\mathbb{C}^{k}$ and let $F_m$ be the $m \times m$ Fourier matrix,
 \[F_m(i, j) = m^{-1/2} \omega^{(i-1)(j-1)}, \quad i, j \in [m],\]
 for $\omega = e^{2\pi i/m}$. Note that $F_m$ is a unitary and every element of $F_m$ has absolute value equal to $m^{-1/2}$. We now define, 
 \[A^{(j)} := U^{j} \otimes F_m \in M_k(\mathbb{C}) \otimes M_m(\mathbb{C}), \quad j \in [k-1].\]
 
 For the final matrix $A^{(k)}$, we use an idea of Casazza et. al. from \cite{CEKP07}. A \emph{conference matrix} $C_m$ in $M_m(\mathbb{C})$ is a zero diagonal Hermitian such that $C_m C_m^{*} = m-1$. As pointed out in the abovementioned reference, conference matrices are known to exist for infinitely many integers $m$. Define the $k$'th matrix in our tuple as, 
 \[A^{(k)} := \dfrac{(I \otimes C_m)}{ \sqrt{m-1}}.\]

A simple calculation shows, 
\begin{proposition}\label{Optimal}
With $A^{(1)}, \cdots, A^{(k)}$ as above, suppose $X_1 \amalg \cdots \amalg X_r = [km]$ is a paving such that, 
\[|| P_{X_i} A^{(j)} P_{X_i} || < \epsilon,\qquad i \in [r], j \in [k],\]
then, we must have that the $X_i$ are all singletons. Consequently, if $(A^{(1)}, \cdots, A^{(k)})$ can be jointly $(r, \epsilon)$ paved, then $r \geq k\lfloor \epsilon^{-2} \rfloor $.

\end{proposition}

\subsection{A weak result concerning lifts of non-bipartite regular graphs}

In \cite{MSS1}, the authors solved a conjecture due to Bilu and Linial \cite{BiLu}, and used this to show the existence of infinitely many $d$ regular bipartite Ramanujan graphs for every $d > 2$. They showed that for every $d$ regular graph, there is a signed adjacency matrix $A_s$ such that $\lambda_{max}(A_s) \leq 2\sqrt{d-1}$. We now make a calculation that is relevant to non-bipartite graphs.

A straightforward calculation yields that the mixed determinantal polynomials $\chi[A_s, -A_s]$ where $s$ runs over all signings of the graph form an interlacing family and that,
\[\mathbb{E}_{s \in \{-1,1\}^{|E|}} \chi[A_s, -A_s](x) = m_G(\sqrt{2} x),\]
where $m_G$ is the matching polynomial of the graph. Using the Heilman-Lieb bound \cite{HeLi}, we see that there is a signing such that, 
\[\lambda_{max} \chi[A_s, -A_s] \leq \dfrac{2\sqrt{d-1}}{\sqrt{2}}.\]
Together with theorem (\ref{thm:shrink}), we deduce that there is a signing such that, 
\[\operatorname{max} \{\lambda_{max} \chi[A_s],-\lambda_{min}\chi[A_s]\} \leq 2\sqrt{2}\sqrt{d-1}.\]
This bound doesn't however yield anything new; Friedman \cite{Fri04} has shown that for any $\epsilon > 0$, with high probability, the non-trivial eigenvalues of any sufficiently large $d$ regular graph lie in $[-2\sqrt{d-1}-\epsilon, 2\sqrt{d-1}+\epsilon]$.

\subsection{Vectorial versions of Anderson Paving}

It is easy to deduce vectorial versions of Anderson paving. Here is one such version; we seek a paving of a block matrix that respects the block structure. We use the expression $\Delta$ to denote the map sending a matrix to its diagonal.
\begin{theorem}
There is a universal constant $C$ such that for any $\epsilon > 0$ and $r, n, k \in \mathbb{N}$, given a matrix $A \in M_k(\mathbb{C}) \otimes M_n(\mathbb{C})$ such that, 
\[(\Delta \otimes I)[A] = 0, \]
there is a collection of diagonal $n \times n$ projections $P_1, \cdots, P_r$ such that $P_1 + \cdots + P_r  = I$ and where $r \leq C k^4 \epsilon^{-2}$ such that, 
\[||(I \otimes P_i ) A (I \otimes P_i )|| < \epsilon \,||A||, \quad i \in [r].\]
If $A \in M_k(\mathbb{C}) \otimes D_n$, where $D_n$ is the diagonal subalgebra, we make take $r \leq D\,k\epsilon^{-2}$ for a universal constant $D$. In this case, the estimate is tight upto constant factors. 
\end{theorem}

Write $A = \left(A_{ij}\right)_{i,j \in [k]}$ for matrices $A_{ij} \in M_n(\mathbb{C})$ and let $X_1 \amalg \cdots \amalg X_s = [n]$ be a $(s, \epsilon/k)$ joint paving for the $k^2$ matrices $\left(A_{ij}\right)_{i,j \in [k]}$. By theorem (\ref{thm:paving}), we may take $s = C k^{4} \epsilon^{-2}$ for a suitable universal constant $C$.

 The desired projections are precisely $P_i = I \otimes P_{X_i}$ for $i \in [s]$. Writing $A_{ij}(l)$ for $P_{X_l}A_{ij}P_{X_l}$, we have that $||(I \otimes P_l) A (I \otimes P_l )|| $ equals the norm of the block matrix $A(l) = \left(A_{ij}(l)\right)_{i, j \in [k]}$. $A(l)$ is a $k \times k$ block matrix with each entry being a square matrix with $\operatorname{Rank}(P_l)$ rows. Further, each entry has norm at most $\epsilon/k$ and it is easy to see that the norm of $A(l)$ is at most $\epsilon$ (for instance by writing it as a sum of $k$ matrices, each with a single block on each row and column). 

For matrices in $M_k(\mathbb{C}) \otimes D_n$, we only need to take a $(s, \epsilon)$ joint paving for $k$ matrices, which explains the better result. In this case, the estimate is asymptotically optimal by proposition (\ref{Optimal}).

\bibliographystyle{amsalpha}

\begin{thebibliography}{99}

\bibitem[And79]{anderson1979extreme}
Joel Anderson, \emph{Extreme points in sets of positive linear maps on
  {B}({H})}, Journal of Functional Analysis \textbf{31} (1979), no.~2,
  195--217.
  
  \bibitem[BiLu]{BiLu}
  Yonatan Bilu and Nathan Linial, \emph{Lifts, Discrepancy and Nearly optimal spectral gap}, Combinatorica \textbf{26} (2006), no.~5, 495--519.

\bibitem[BB08]{BBJohnson}
Julius Borcea and Petter Branden, \emph{Applications of stable polynomials to
  mixed determinants: Johnson's conjectures, unimodality, and symmetrized
  fischer products}, DukeMathematical Journal \textbf{143} (2008), no.~2,
  205--223.

\bibitem[BB10]{BBWeylAlgebra}
Julius Borcea and Petter Br\"{a}nd\'{e}n, \emph{Multivariate {Polya-Schur}
  classification problems in the {Weyl} algebra}, Proceedings of the London
  Mathematical Society \textbf{(3) 101} (2010), no.~1, 73--104.

\bibitem[BT87]{BT87}
J.~Bourgain and L.~Tzafriri, \emph{Invertibility of ``large'' submatrices with
  applications to the geometry of {B}anach spaces and harmonic analysis},
  Israel J. Math. \textbf{57} (1987), no.~2, 137--224. 
  
\bibitem[CEKP07]{CEKP07}
Pete Casazza, Dan Edidin, Deepti Kalra, and Vern~I. Paulsen, \emph{Projections
  and the {K}adison-{S}inger problem}, Operators and Matrices \textbf{1}
  (2007), no.~3, 391–408.

\bibitem[Fel80]{Fell}
H.~J. Fell, \emph{Zeros of convex combinations of polynomials}, Pacific J.
  Math. \textbf{89} (1980), no.~1, 43--50.
  
  \bibitem[Fri04]{Fri04}
  Joel Friedman, \emph{A Proof of Alon's Second Eigenvalue Conjecture and Related Problems}, Mem. of the AMS, \textbf{910} (2004).
  
  \bibitem[HeLi]{HeLi}
  Ole~J. Heilman and Elliot~H.Lieb, \emph{Theory of Monomer-dimer systems}, Comm. Math. Phy. \textbf{25} (1972) no.~1, 190--232.

\bibitem[JOS13]{JOS}
William~B. Johnson, Narutaka Ozawa, and Gideon Schechtman, \emph{A quantitative
  version of the commutator theorem for zero trace matrices}, Proc. Natl. Acad.
  Sci. USA \textbf{110} (2013), no.~48, 19251--19255.

\bibitem[MSS13]{MSS1}
Adam~W. Marcus, Daniel~A Spielman, and Nikhil Srivastava, \emph{Interlacing
  families i: Bipartite {R}amanujan graphs of all degrees}, Foundations of
  Computer Science (FOCS), 2013 IEEE 54th Annual Symposium on, IEEE, 2013,
  pp.~529--537.

\bibitem[MSS14]{MSSICM}
Adam~W. Marcus, Daniel~A. Spielman, and Nikhil Srivastava, \emph{Ramanujan
  graphs and the solution of the {K}adison-{S}inger problem}, Proc. ICM
  \textbf{3} (2014), 375--386.

\bibitem[MSS15]{MSS2}
Adam~W. Marcus, Daniel~A. Spielman, and Nikhil Srivastava, \emph{Interlacing families {II}: {M}ixed characteristic polynomials
  and the {K}adison-{S}inger problem}, Ann. of Math. (2) \textbf{182} (2015),
  no.~1, 327--350.

\bibitem[MSS17]{IF3}
Adam Marcus, Daniel Spielman, and Nikhil Srivastava, \emph{Interlacing families
  iii: Restricted invertibility}, in preparation (2017).

\bibitem[NY16]{naor2016restricted}
Assaf Naor and Pierre Youssef, \emph{Restricted invertibility revisited}, arXiv
  preprint arXiv:1601.00948 (2016).

\bibitem[PV15]{popa2015paving}
Sorin Popa and Stefaan Vaes, \emph{Paving over arbitrary masas in von neumann
  algebras}, Analysis \& PDE \textbf{8} (2015), no.~4, 1001--1023.

\bibitem[Rav16]{MRKSP}
Mohan Ravichandran, \emph{Mixed determinants and the kadison-singer problem},
  https://arxiv.org/abs/1609.04195 (2016).

\bibitem[Rav17]{MRI}
Mohan Ravichandran, \emph{Principal {S}ubmatrices, {R}estricted {I}nvertibility and a
  {Q}uantitative {G}auss-{L}ucas theorem}, https://arxiv.org/abs/1609.04187
  (2017).
	
	\bibitem[ROS56]{rosenblum}
	Marvin Rosenblum, \emph{On the operator equatşon $BX-XA=Q$}, 	Duke. Math. J. \textbf{23} (1956), 263 -- 269.

\bibitem[SS12]{SS12}
Daniel~A Spielman and Nikhil Srivastava, \emph{An elementary proof of the
  restricted invertibility theorem}, Israel Journal of Mathematics \textbf{190}
  (2012), no.~1, 83--91.

\bibitem[Ver01]{vershynin2001john}
Roman Vershynin, \emph{John's decompositions: selecting a large part}, Israel
  Journal of Mathematics \textbf{122} (2001), no.~1, 253--277.

\bibitem[You14]{YoussefIMRN}
Pierre Youssef, \emph{A note on column subset selection}, Int. Math. Res. Not.
  IMRN (2014), no.~23, 6431--6447.
\end{thebibliography}

\end{document}